\documentclass[12pt, reqno]{amsart}

\usepackage{amssymb,latexsym,amsmath,amsfonts}
\usepackage{enumitem}
\usepackage{mathrsfs}
\usepackage{graphicx}
\usepackage{hyperref}
\usepackage[usenames]{color}

\DeclareFontFamily{OT1}{rsfs}{}
\DeclareFontShape{OT1}{rsfs}{n}{it}{<-> rsfs10}{}
\DeclareMathAlphabet{\mathscr}{OT1}{rsfs}{n}{it}

\numberwithin{equation}{section}

\theoremstyle{definition}
\newtheorem{definition}{Definition}[section]

\theoremstyle{remark}
\newtheorem{remark}[definition]{Remark}

\theoremstyle{plain}
\newtheorem{theorem}[definition]{Theorem}
\newtheorem{result}[definition]{Result}
\newtheorem{lemma}[definition]{Lemma}
\newtheorem{proposition}[definition]{Proposition}
\newtheorem{example}[definition]{Example}
\newtheorem{corollary}[definition]{Corollary}

\definecolor{DPurple}{rgb}{0.76,0.2,0.69}

\newcommand{\J}{{\rm J}}

\newcommand{\f}{F}

\newcommand\decosum[1]{\sum\nolimits_{{#1}}^{'}}
\newcommand{\countsum}{\sideset{}{'}\sum}
\newcommand{\supp}{{\rm supp}}

\newcommand{\D}{D}
\newcommand{\dg}{\mathrm{d}}

\newcommand{\sph}{\widehat{\mathbb{C}}}
\newcommand{\C}{\mathbb{C}} 
\newcommand{\R}{\mathbb{R}}
\newcommand{\Z}{\mathbb{Z}}
\newcommand{\N}{\mathbb{N}}
\newcommand{\Q}{\mathbb{Q}}

\begin{document}

\title[Ergodicity in the dynamics of holomorphic correspondences]{Ergodicity 
in the dynamics of holomorphic correspondences}

\author{Mayuresh Londhe}
\address{Department of Mathematics, Indiana University, Bloomington, Indiana 47405, USA}
\email{mmlondhe@iu.edu}

\begin{abstract}
This paper studies ergodic properties of certain measures arising in the dynamics of holomorphic
correspondences.
These measures, in general, are not invariant in the classical sense of ergodic theory.
We define a notion of ergodicity, and prove a version of Birkhoff's ergodic theorem
in this setting. 
%In fact, we strengthen
%this classical result in the setting of rational maps on the Riemann sphere with the Lyubich measure.
We also show the existence of ergodic measures when a holomorphic 
correspondence is defined 
on a compact complex manifold.
Lastly, we give an explicit class of dynamically interesting measures that are ergodic as in our definition.
\end{abstract}


\maketitle

\vspace{-0.25cm}
\section{Introduction and statement of main results}\label{S:intro}

In this paper, we study ergodic properties of certain measures arising naturally 
in iterative holomorphic dynamics beyond the classical framework of maps.
Loosely speaking, ergodicity expresses the idea that a \emph{typical} point of a dynamical system
will eventually visit all parts of the space that 
the system moves in, in a uniform and random way.
The measures under consideration, in general,
are not invariant in the classical sense of ergodic theory.
Thus the classical ergodic theorems do not hold as it is in this setting.
The purpose of this article is to define a notion of ergodicity and to prove an
ergodic theorem for such measures.
\smallskip

In \cite{brolin:isuirm65}, Brolin constructed a probability measure supported on the Julia set of a polynomial 
of degree at least 2. 
This measure describes the distribution of the preimages of a generic point under iteration of the polynomial.
Freire, Lopes and Ma{\~n}{\'e} \cite{FLM:imrm83}, and Lyubich \cite{ljubich:eprers83}, independently generalised 
this to all
%the case of 
rational maps of degree at least 2 on the Riemann sphere. These measures are invariant and ergodic 
in the sense of ergodic theory\,---\,see, for instance, \cite{Walters:aiet82} for the basics of ergodic theory.
We refer the reader to \cite{DinhSibony:discv10} and the references therein for various 
generalisations of the above results,
and for dynamically interesting properties possessed by such measures.
Also, see \cite[Chapter~13]{Ben:donav19} and the references therein for analogous measures in
non-Archimedean setting.
Some of these constructions have been further extended to certain classes of multi-valued maps.
But, the resulting measures are no longer invariant. Thus
the classical ergodic theorems do not apply in this setting.
However, it turns out that these measures are invariant in a certain sense.
We now proceed to formally define the sense in which these measures are invariant and the 
dynamical systems in which they arise.
\smallskip

Let $X$ be a (not necessarily compact) complex manifold of dimension $k$.
We shall always assume that manifolds are connected.
A \emph{holomorphic $k$-chain} is a formal linear combination of the form
\begin{equation}\label{E:corresp}
  \Gamma= \sum_{1 \leq i \leq N} m_i\Gamma_i,
\end{equation} 
where the $m_i$'s are positive integers and $\Gamma_i$'s are distinct irreducible complex subvarieties of
$X \times X$ of dimension $k$. Let $\pi_1$ and $\pi_2$ denote the projection onto the first and second
coordinates respectively
and let $|\Gamma|:= \cup_{i=1}^N\Gamma_i$.
We call $\Gamma$ a \emph{holomorphic correspondence on $X$} if 
$\pi_1|_{\Gamma_i}$ and $\pi_2|_{\Gamma_i}$ are proper, surjective and finite-to-one 
maps for each $1 \leq i \leq N$.
If $A$ is a subset of $X$, then we define the following set-valued maps
\[
  \f_{\Gamma}(A):= \pi_2 (\pi_1^{-1} (A) \cap |\Gamma|) {\rm{\ and \ }} \f_{\Gamma}^{\dagger}(A)
  := \pi_1 (\pi_2^{-1} (A) \cap |\Gamma|).
\]
For convenience, we denote $\f_{\Gamma}(\{x\})$ and $\f_{\Gamma}^{\dagger}(\{x\})$ by $\f_{\Gamma}(x)$
and $\f_{\Gamma}^{\dagger}(x)$ respectively. 
When there is no scope for confusion, we shall, for simplicity of notation, denote $F_{\Gamma}$ by
$F$. Also, we shall \emph{refer to the correspondence $\Gamma$ underlying $F$ also as $F$}.
\smallskip

In Section~\ref{S:def}, the pullback of a Borel measure $\mu$ by $\f$, denoted by $\f^*\mu$,
is introduced in detail. This operation plays a key role in defining a notion of invariance for measures. 
We say that a Borel probability measure $\mu$ is \emph{$\f^*$-invariant} if 
$\f^*\mu= \dg \cdot \mu$,
where $\dg$ is the topological degree of $\f$\,---\,see
Section~\ref{S:def} for a definition.
Given an $\f^*$-invariant measure $\mu$,
a Borel subset $B$ of $X$ is said to be \emph{almost invariant} with respect to $\f$ and $\mu$
if there exists a Borel set $B' \subseteq B$ such that $\f^{\dagger}(B') \subseteq B$
and $\mu(B')=\mu(B)$. We now define a notion of ergodicity for $\f^*$-invariant measures.

\begin{definition}\label{D:ergodic}
Let $\f$ be a holomorphic correspondence on a complex manifold $X$.
We say that an $\f^*$-invariant Borel probability measure $\mu$ on $X$ is \emph{ergodic} 
if for any Borel set $B$ that is almost invariant with respect to $\f$ and $\mu$,
either $\mu(B)=0$ or $\mu(B)=1$.
\end{definition}

\smallskip

Two holomorphic correspondences on $X$ can be composed with each other\,---\,see
Section~\ref{S:def} for the definition. Keeping in mind the above notational comments,
we shall write $\f^n$ to denote the $n$-fold composition
of a holomorphic correspondence $\f$ on $X$. Thus $\f$ gives rise to a dynamical system on $X$.
We now state the first main result of this paper, which is an analogue of the classical Birkhoff's
Ergodic Theorem in the setting of the above dynamical system.
\begin{theorem}\label{T:ergodic}
Let $\f$ be a holomorphic correspondence of topological degree $\dg$ on a complex manifold
$X$. Suppose there exists
a Borel probability measure $\mu$ on $X$ such that $\mu$ is 
$\f^*$-invariant, i.e., it satisfies $\f^*\mu= \dg \cdot \mu$.
Then, for $\phi \in L^1(\mu)$,
\[
\frac{1}{n}\sum_{j=0}^{n-1}
\countsum_{y \in {F^j}^{\dagger}(x)} \frac{\phi(y)}{\dg^j}
\]
converges $\mu$-almost everywhere to a function $\Phi \in L^1(\mu)$ such that
$\int_X \Phi\,d\mu=\int_X \phi\,d\mu$.
Moreover, if $\mu$ is ergodic as in Definition~\ref{D:ergodic} then, we have
\[
\Phi(x)=\int_X \phi\,d\mu \quad \text{for $\mu$-almost every $x\in X$}.
\]
\end{theorem}

\noindent{In Theorem~\ref{T:ergodic}, the notation ${\sum\nolimits}'$ denotes the sum over $y$'s, repeated
with multiplicity. 
\begin{remark}
Since an $\f_*$-invariant measure is also an ${(\f^{\dagger})}^*$-invariant measure,
Theorem~\ref{T:ergodic} also holds for $\f_*$-invariant measures, 
with ${\f^j}^{\dagger}$ replaced by $\f^j$ in the ergodic sums.
Note that whenever $\f$ is a holomorphic map (i.e., the correspondence defined by the graph of 
a holomorphic map), the $\f_*$-invariance of $\mu$ in the above sense coincides with the invariance
of $\mu$ in the sense of ergodic theory for measurable maps. Additionally, in general, $\f^*$-invariance
does not imply $\f_*$-invariance unlike in the case of holomorphic maps.
\end{remark}

A few comments are in order for the above notion of ergodicity.
If $B$ is a subset of $X$
such that $\f^{\dagger}(B) \subseteq B$,
the complement $B^c$ of $B$ may not satisfy $\f^{\dagger}(B^c) \subseteq B^c$
as in the case of maps.
In Section~\ref{S:existence}, we shall see that if a Borel subset $B$ is almost invariant 
with respect to $\f$ and $\mu$, then
$B^c$ is also almost invariant with respect to $\f$ and $\mu$. However,
for such Borel sets $B$,
the symmetric difference $B \triangle \f^{\dagger}(B)$ may not satisfy $\mu (B \triangle \f^{\dagger}(B))=0$
as in the case of maps\,---\,see Example~\ref{Ex:sym_diff}.
Theorem~\ref{T:ergodic} immediately raises a question: do there exist $\f^*$-invariant measures
for correspondences that are not maps?
In Section~\ref{S:existence}, we show that an $\f^*$-invariant measure always exists for a
holomorphic correspondence $\f$ defined on a compact complex manifold.
In fact, we show that an $\f^*$-invariant ergodic measure exists for such correspondences.
The following result of Dinh--Sibony
gives a class of holomorphic correspondences admitting dynamically interesting
measures that are $\f^*$-invariant:

\begin{result}[Dinh--Sibony, \cite{DinhSibony:dvtma06}]\label{R:DS}
Let $(X, \omega)$ be a compact K\"{a}hler manifold of dimension $k$.
Let $\f$ be a meromorphic correspondence of topological degree $\dg$ on
$(X,\omega)$. Suppose that the dynamical degree of order $k-1$, denoted $d_{k-1}$,
satisfies $\dg_{k-1} < \dg$. Then, the measures $\dg^{-n} {(\f^n)}^* {\omega}^k$
($\omega$ normalised so that $\int_X\omega^k=1$)
converges to a Borel probability measure $\mu_F$ as $n \to \infty$. Moreover, $\mu_F$ does not put any mass on pluripolar sets
and $\mu_F$ is $\f^*$-invariant.
\end{result}

When $\f$ is a rational map on the Riemann sphere, the measure $\mu_\f$ in Result~\ref{R:DS} 
is the measure constructed in \cite{brolin:isuirm65, FLM:imrm83, ljubich:eprers83}.
We shall call the measure $\mu_F$ given by Result~\ref{R:DS} 
the \emph{Dinh--Sibony measure of} $F$\,---\,see Section~\ref{S:def} for more details.
Here, we would like to mention few more classes of correspondences for which dynamically interesting invariant
($\f^*$-invariant or $\f_*$-invariant)
measures exist.
In \cite{Dinh:ddppdcp05}, Dinh constructed an $\f^*$-invariant measure when $\f$ is a polynomial 
correspondence whose Lojasiewicz exponent is strictly greater than 1. 
Clozel and Otal \cite{CloOtal:uecm01} and Clozel and Ullmo \cite{CloUllmo:cmmi03} constructed invariant 
measures for certain classes of modular correspondences.
On the other hand, Dinh, Kaufmann and Wu \cite{DinhKaufWu:dhcrs20} 
constructed invariant measures for holomorphic correspondences on Riemann surfaces
 that are {\bf{not}} weakly-modular. Recently, Matus de la
Parra \cite{Parra:emqmmg22} constructed invariant measures for the family of correspondences
considered by Bullett and Lomonaco in \cite{BullLomo:mqmwmg20}. 
Lastly, Bharali and Sridharan \cite{BhaSri:dhc16} constructed invariant measures for correspondences 
having a repeller.
\smallskip

Having mentioned some examples of invariant measures, we now move to ergodicity.
Our next result asserts that the measures given by Result~\ref{R:DS} are ergodic.

\begin{theorem}\label{T:ergodicity_DS}
Let $(X, \omega)$ be a compact K\"{a}hler manifold of dimension $k$.
Let $\f$ be a holomorphic correspondence of topological degree $\dg$ on
$X$. Suppose that $\dg_{k-1} < \dg$. 
Then the Dinh--Sibony measure $\mu_\f$ is ergodic as in Definition~\ref{D:ergodic}.
\end{theorem}

\noindent{Even though, we have stated Theorem~\ref{T:ergodicity_DS} for holomorphic correspondences,
Theorem~\ref{T:ergodicity_DS} (and Theorem~\ref{T:ergodic}) holds for meromorphic correspondences
as in Result~\ref{R:DS}.
The property of $\mu_{\f}$ that $\mu_{\f}$ puts zero mass on pluripolar sets plays an important role here 
(see Remark~\ref{Re:mero_corres} for further details).}
We would also like to mention that the proof of Theorem~\ref{T:ergodic} is 
purely measure theoretic, and thus holds for 
certain multi-valued maps on more general spaces.
As mentioned earlier, if $f$ is a rational map of degree at least 2, then the Dinh--Sibony measure $\mu_f$
is the measure
constructed in \cite{FLM:imrm83, ljubich:eprers83}.
This, combined with Theorem~\ref{T:ergodic} and Theorem~\ref{T:ergodicity_DS},
immediately gives the following ergodic theorem that is
new in the classical case of rational maps:
%different than what is known in the literature:
\begin{corollary}\label{C:maps}
Let $f$ be a rational map on the Riemann sphere $\sph$ of degree $\dg \geq 2$. Let
$\mu_f$ be the measure constructed in \cite{FLM:imrm83, ljubich:eprers83}.
Then, for $\phi \in L^1(\mu_f)$, we have
\[
\lim_{n\to \infty}\frac{1}{n}\sum_{j=0}^{n-1}
\countsum_{y \in {f^{-j}}(x)} \frac{\phi(y)}{\dg^j}
= \int_{\sph} \phi\,d\mu_f
\]
for $\mu_f$-almost every $x \in \sph$.
\end{corollary}
%\noindent{If we replace $\phi$ by $\phi \circ f^n$ in Corollary~\ref{C:maps}, 
%since $\int_{\sph} \phi\,d\mu_f=\int_{\sph} \phi \circ f^n \,d\mu_f$,
%we get the classical Birkhoff's ergodic theorem for the measure $\mu_f$.}
%\smallskip

We end this section by mentioning a concrete class of correspondences, 
where Theorem~\ref{T:ergodic} can be useful.
Let $\mathcal G=\{f_1, \dots, f_N\}$ be a finite collection of surjective endomorphisms of $X$. 
To the collection $\mathcal G$, we associate a correspondence as follows:
\begin{equation}\label{E:semi_corres}
  \Gamma_{\mathcal G}:= \sum_{1 \leq i \leq N} {\rm{graph}} (f_i),
\end{equation}
where ${\rm{graph}} (f_i)$ is the graph of $f_i$. 
If $X$ is the Riemann sphere and $f_i$'s are rational maps of degree at least 2, 
then the measure $\mu_{\mathcal G}$ (the
Dinh--Sibony measure associated with the correspondence $\Gamma_{\mathcal G}$) coincides with 
the measure constructed by Boyd \cite{boyd:imfgrs99}. Since, in this case, the support of $\mu_{\mathcal G}$ is
equal to the Julia set of the semigroup generated by ${\mathcal G}$, the Theorem~~\ref{T:ergodic}
can be useful to understand the dynamics on the Julia set.
The idea of studying the dynamics of a finitely generated rational semigroup through the correspondence
$\Gamma_{\mathcal G}$ was introduced by Bharali--Sridharan in \cite{BhaSri:hcrfgrs17}.
The other special case is when $f_i$'s are the M\"{o}bius transformations. Studying dynamics of 
$\Gamma_{\mathcal G}$ in this case can be used to study finitely generated Kleinian groups.
The results in \cite{DinhKaufWu:dhcrs20} can be used to construct invariant measures for
 non-elementary finitely generated Kleinian groups. Also, see \cite{DinhKaufWu:prmdpv21} for an
 application of correspondences to random matrices.
\smallskip
 
\noindent{\bf{Outline of the paper:}} Section~\ref{S:def} gives background on holomorphic correspondences
and the Dinh--Sibony 
measure. In section~\ref{S:existence}, we give existence of $\f^*$-invariant and ergodic measures for a
holomorphic correspondence defined on a compact complex manifold. The proof of 
Theorems~\ref{T:ergodicity_DS}
appears in Section~\ref{S:examples}.
Section~\ref{S:prelem} is devoted to proving preliminary results needed for the
proof of Theorem~\ref{T:ergodic}. In the last section, Section~\ref{S:ergodic}, we give the proof of 
Theorem~\ref{T:ergodic}.
\medskip

\section{Fundamental definitions}\label{S:def}

In this section, we shall collect some definitions and concepts about holomorphic correspondences
that we had mentioned in passing in Section~\ref{S:intro}. 
We refer the reader to \cite{DinhSibony:dvtma06} for a more detailed discussion in the setting
of meromorphic correspondences defined on compact manifolds.
Most of the material in this section is well known; 
the reader familiar with these concepts can safely move on to the next section. 
\smallskip

Let $X$ be a complex manifold, let $\Gamma$ be a holomorphic correspondence on $X$, and
let $\f$ and $\Gamma$ be related as described in Section~\ref{S:intro}.
With the presentation of $\Gamma$ as in \eqref{E:corresp},
the coefficient $m_i\in \Z_+$ will be called the \emph{multiplicity} of $\Gamma_i$.
We shall call $\Gamma$ the \emph{graph} of $\f$. 
We define the \emph{support} of the correspondence $\f$ by $|\Gamma|:=\cup_{i=1}^{N} \Gamma_i$.
For $\Gamma_i$ as above, we define
$\Gamma_i^{\dagger}:=\{ (y,x) : (x,y) \in \Gamma_i\}$. Now we use this to define
the \emph{adjoint}
\[
  \Gamma^{\dagger}:= \sum_{1 \leq i \leq N} m_i\Gamma_i^{\dagger}.
\]
Observe that $\Gamma^{\dagger}$ is also a holomorphic correspondence on $X$.
The holomorphic correspondence $\Gamma^{\dagger}$ is called the \emph{adjoint} of $\Gamma$.
Also, note that the set valued map $\f_{\Gamma}^{\dagger}$, defined in Section~\ref{S:intro}, is same as
the set-valued map $\f_{\Gamma^{\dagger}}$ induced by
$\Gamma^{\dagger}$.
We shall adopt the notational convenience noted in Section~\ref{S:intro}
and refer to the correspondence $\Gamma^{\dagger}$ as $\f^{\dagger}$.
\smallskip

The \emph{topological degree} of $\f$ is the number of points in a generic fiber counted with multiplicities.
It is well known that there exists a non-empty Zariski-open set $\Omega \subseteq X$ such that,
writing $Y^i:= \pi_2^{-1}(\Omega)\cap \Gamma_i$, the
map $\left.\pi_2\right|_{Y^i}: Y^i \to \Omega$ is a $\delta_i$-sheeted holomorphic covering for
some $\delta_i \in \Z_+$, $i=1, \dots , N$. 
Thus $\delta_i=\sharp\{y: (y,x) \in \Gamma_i\}$ for any $x \in \Omega$,
where $\sharp$ denotes the cardinality. 
Then the topological degree $\dg(\f)$ of $\f$ is 
\begin{equation}\label{E:topdeg}
  \dg(\f):= \sum_{i=1}^N m_i \delta_i = \sum_{i=1}^N m_i \ \sharp\{y: (y,x) \in \Gamma_i\}, \quad x \in \Omega.
\end{equation}
We shall use $\dg$ instead of $\dg(\f)$ whenever there is no confusion.
It is classical that, for every $x \in X$, $\f_{\Gamma}^{\dagger}(x)$ contains $\dg$ points counted with multiplicity.
\smallskip

Two holomorphic correspondences on $X$ can be composed to get a new holomorphic correspondence on $X$.
Let $\f_1$ and $\f_2$ be two holomorphic correspondences on $X$ induced by holomorphic $k$-chains
\[
  \Gamma^1 = \countsum_{1\leq i\leq M}\Gamma^1_i  {\rm{ \ and \ }}
  \Gamma^2 = \countsum_{1 \leq j \leq M'} \Gamma^2_j
\]
respectively. In the above presentation of $\Gamma^1$ (resp., $\Gamma^2$), we do not assume that the
$\Gamma^1_i$'s (resp., $\Gamma^2_j$'s) are distinct varieties\,---\,varieties repeat according to multiplicities.
The ``decorated'' summation above will denote the latter presentation.
Then, by definition, $\f_1 \circ \f_2$ is the holomorphic correspondence induced by
\[
  \Gamma^1  \circ \Gamma^2 = \countsum_{1 \leq i \leq M} \countsum_{1 \leq j \leq M'}\Gamma^1_i  \circ \Gamma^2_j,
\]
where $\Gamma^1_i  \circ \Gamma^2_j$ is defined as follows:
Consider the subset of $X \times X$ given by
\begin{equation}\label{E:compos}
C:=  \{(x_1,x_3) \in X \times X : \exists x_2 \in X {\rm{ \ such \ that \ }}
  (x_1,x_2) \in \Gamma^2_j {\rm{ \ and \ }} (x_2,x_3) \in \Gamma^1_i\}.
\end{equation}
The composition $\Gamma^1_i  \circ \Gamma^2_j$ is the holomorphic $k$-chain whose support is $C$.
Note that the set $C$ need not be an irreducible subvariety.
The multiplicities of irreducible components of $C$ are as follows. 
Let $C_{s}$ denote an arbitrary
irreducible component of $C$. Then, its multiplicity in $\Gamma^1_i  \circ \Gamma^2_j$
is the distinct number of $x_2$'s satisfying the condition stated in \eqref{E:compos} for a
\textbf{generic} $(x_1,x_3)\in C_{ij,\,s}$.
We would like to emphasise that $\Gamma^1_i  \circ \Gamma^2_j$ need not be irreducible.
This is the reason why the data defining a holomorphic correspondence must include multiplicities.
With the above definition of composition, if $A$ is a subset of $X$ then we have
 $\f_1 \circ \f_2(A)=\f_1(\f_2(A))$.
\smallskip

If $\f$ is a holomorphic correspondence on $X$ and $B$ is a Borel subset of $X$.
We show that $\f(B)$ and $\f^{\dagger}(B)$ are Borel subsets of $X$. To show this we need a
result about bimeasurable functions.
Let $X$ and $Y$ be topological spaces. Recall that
a function $f:X \to Y$ is Borel measurable if the preimage (under $f$) of every Borel subset of $Y$
is a Borel subset of $X$.
We say that a Borel measurable function $f$ is \emph{bimeasurable} if the image (under $f$) of every
Borel subset of $X$ is a Borel subset of $Y$.

\begin{result}[Purves, \cite{Purves:bf66}]\label{R:bimeasurable}
Let $X$ and $Y$ be complete separable metric spaces and $E$ a Borel subset of $X$.
Consider a Borel measurable function $f:E \to Y$. 
In order that $f$ is bimeasurable it is necessary and sufficient that the set
$\{\zeta \in Y: f^{-1}{\{\zeta\}} \textnormal{ is uncountable}\}$ is countable. 
\end{result}

For an alternative proof of the above result, also see \cite{Mauldin:bf81}. 
Observe that Result~\ref{R:bimeasurable} holds for connected complex manifolds.
By definition
$\f(B):= \pi_2 (\pi_1^{-1} (B) \cap |\Gamma|)$. Since $|\Gamma|$ is a closed set,
$\pi_1^{-1} (B) \cap |\Gamma|$ is a Borel subset of $X \times X$. By the definition
of a holomorphic correspondence, $\pi_2|_{\Gamma_j}$ is a finite map for each $j$.
Thus, by Result~\ref{R:bimeasurable}, it follows that $\f(B)$ is a Borel subset of $X$.
Similarly, it follows that $\f^{\dagger}(B)$ is also a Borel subset of $X$.
\smallskip

Let $\D$ be a current on $X$ of bidegree $(p,p)$, $0 \leq p \leq k$. We can pull back and
push forward $\D$ using the following prescription:
\begin{equation}\label{E:pull_current}
  \f^*(\D):= (\pi_1)_*(\pi_2^*\D\wedge [\Gamma])  {\rm{\ and \ }}  \f_*(\D):= (\pi_2)_*(\pi_1^*\D\wedge [\Gamma]) 
\end{equation}
whenever the intersection current $\pi_2^*\D\wedge [\Gamma]$ and $\pi_1^*\D\wedge [\Gamma]$ makes sense. 
Here, $[\Gamma]$ denotes the sum (weighted by multiplicities) of the currents of integration on   
the varieties that constitute $\Gamma$.
In this paper, we are mainly interested in the pull-back of a finite Borel measure\,---\,which can be viewed as
a current of bidegree $(k,k)$. Let $\nu$ be a finite Borel measure on $X$. We will work out
$\f^*\nu$ in detail here. Let $\phi$ be a compactly supported continuous function on $X$.
\[
  \langle \f^*\nu , \phi \rangle
  =\langle \pi_2^*(\nu)\wedge [\Gamma], \pi_1^*\phi \rangle
  :=\sum_{1 \leq i \leq N} m_i \langle \nu, (\pi_2|_{\Gamma_i})_* (\phi \circ \pi_1) \rangle.
\]
Let $\Omega \subseteq X$ be a Zariski-open set and let $(Y^i, \Omega, \left.\pi_2\right|_{Y^i})$
be the holomorphic coverings introduced prior to \eqref{E:topdeg}
for each $ i=1, \dots , N$. Then for each $x \in \Omega$,
$(\pi_2|_{\Gamma_i})_* (\phi \circ \pi_1)(x)$ is the sum of the values of 
$\phi \circ \pi_1$ on the fiber $\pi_2^{-1}\{x\} \cap \Gamma_i$.
Thus we get
\[
  (\pi_2|_{\Gamma_i})_* (\phi \circ \pi_1)(x)= \sum_{y: (y,x) \in \Gamma_i} \phi(y).
\]
It is classical that $(\pi_2|_{\Gamma_i})_* (\phi \circ \pi_1)(x)$ extends continuously to 
$X \setminus \Omega$. Thus, for $x \in X$, we get
\[
(\pi_2|_{\Gamma_i})_* (\phi \circ \pi_1)(x)= \countsum_{y \in \f^{\dagger}(x)} \phi(y),
\]
where ${\sum\nolimits}'$ denotes the sum with $y$'s, repeated with multiplicity.
Therefore, we have
\[
  \langle \f^*\nu, \phi \rangle := \int_X \countsum_{y \in \f^{\dagger}(x)} \phi(y)\,d\nu(x).
\]
If $\nu$ is a Borel probability measure
then $\f^*\nu$ is a positive measure of total mass equal to the topological degree of $\f$.
Let $\f$ be a holomorphic correspondence of topological degree $\dg$ on $X$ and 
$\mu$ an $\f^*$-invariant measure on $X$,
i.e., it satisfies $\f^*\mu= \dg \cdot \mu$.
Using the definitions above, this is equivalent to:
for any compactly supported continuous function 
(more generally, by density, for any $\mu$-integrable function)
$\phi$ on $X$, we have
\begin{equation}\label{E:inv_int}
\frac{1}{\dg}\int_X \countsum_{y \in \f^{\dagger}(x)} \phi(y)\,d\mu(x) = \int_X \phi\,d\mu.
\end{equation}
Let $\f_1$ and $\f_2$ be two meromorphic correspondences on $X$ and $\nu$ be a
probability measure on $X$. It easily follows from the above discussion that
$ (\f_1 \circ \f_2)^*\nu = {\f_2}^*({\f_1}^*\nu)$.
Thus, if $\mu$ is $\f^*$-invariant then $\mu$ is ${(\f^n)}^*$-invariant for every $n \in N$.
\smallskip

We end this section by discussing the Dinh--Sibony measure associated with certain holomorphic
correspondences, which we had mentioned in Section~\ref{S:intro}.
To discuss the existence of the Dinh--Sibony measure for a holomorphic correspondence, 
we need to define the pull-back of a smooth $(p,p)$-form.
Let $\f$ be a holomorphic correspondence on a compact complex manifold of dimension $k$. Consider
a smooth $(p,p)$-form $\alpha$, $0 \leq p \leq k$.
Since $\alpha$ is also a current of bidegree $(p,p)$, the prescription \eqref{E:pull_current}
defines $\f^*(\alpha)$, since $\pi_2^* \alpha \wedge [\Gamma]$ makes sense as a $(p,p)$-current
on $|\Gamma|$.
\smallskip

Consider a compact K\"{a}hler manifold $(X, \omega)$ of dimension $k$, and let $\int {\omega}^k=1$.
Consider a holomorphic correspondence $\f$ on $X$ of topological degree $\dg$.
We define the dynamical degree of order $p$, $0 \leq p \leq k$,
\[
  \dg_{p}(\f):= \lim_{n \to \infty} {\left( \int_X {(\f^n)}^* {\omega}^{p} \wedge {\omega}^{k-p} \right)}^{1/n}.
\]
Note that $\dg_k(\f)$ is just the topological degree of $\f$.
We shall use $\dg_p$ instead of $\dg_p(\f)$ whenever there is no confusion.
We now define
a sequence $\mu_n:=\dg^{-n} {(\f^n)}^* {\omega}^k$.
Since ${\omega}^k$ is a volume form on $X$, it follows that $\mu_n$ is a sequence
of probability measures.
Under the hypothesis that
$\dg_{k-1} < \dg$, Dinh--Sibony proved \cite[Section~5]{DinhSibony:dvtma06}
that $\mu_n$ converges in the weak* topology to a $\f^*$-invariant probability measure $\mu_{\f}$.
In fact, they showed that if $u$ is a quasi-p.s.h. function (a function that is locally the sum of a
smooth function and a plurisubharmonic function) then $u$ is $\mu_{\f}$-integrable and
$\langle\mu_n, u\rangle \to \langle\mu_{\f}, u\rangle$ as $n \to \infty$.
%This is the mode of convergence implicit in Result~\ref{R:DS}.
In particular, $\mu_{\f}$ puts zero mass on pluripolar sets.
Dinh--Sibony also showed that preimages of a generic point are equidistributed according to the measure $\mu_{\f}$,
i.e., there exists a pluripolar subset $E$ of $X$ such that for every $a \in X \setminus E$, we have
\[
  \dg^{-n} {(\f^n)}^* {\delta}_a\to \mu_{\f}
\] 
as $n \to \infty$. See \cite[Sections~1 and~5]{DinhSibony:dvtma06} for a detailed discussion. We shall
use the above properties
in Section~\ref{S:examples} to prove the ergodicity of $\mu_{\f}$.
\smallskip

We end this section by mentioning that the Result~\ref{R:DS} and the above equidistribution property hold
for a meromorphic correspondence (satisfying the degree condition in Result~\ref{R:DS}) on a compact 
K\"{a}hler manifold. See \cite{DinhSibony:dvtma06} for more details. 
Also, see a recent article by Vu \cite{Vu:emmskm20} for an extension of these results
to meromorphic correspondences on compact non-K\"{a}hler manifolds.
\medskip

\section{Existence of $\f^*$-invariant and ergodic measures}\label{S:existence}

Let $\f$ be a holomorphic correspondence on a complex manifold $X$.
This section is devoted to proving the existence of an $\f^*$-invariant and an $\f^*$-invariant 
ergodic measure on $X$ when $X$ is compact.
Before proving these results, we prove a lemma about the complement of an almost invariant set. 
Recall that a Borel subset $B$ of $X$ is almost invariant with respect to $\f$ and $\mu$ 
if there exists a Borel set $B' \subseteq B$
such that $\f^{\dagger}(B') \subseteq B$ and $\mu(B')=\mu(B)$.
When $\f$ and $\mu$ are clear from the context, for simplicity, we shall just say that $B$ is almost invariant.

\begin{lemma}\label{L:alminv_complement}
Let $\f$ be a holomorphic correspondence of topological degree $\dg$ on a complex manifold
$X$ and $\mu$ an $\f^*$-invariant Borel probability measure on $X$.
Let $B$ be a Borel subset of $X$ such that $B$ is almost invariant with respect to $\f$ and $\mu$.
Then $B^c$ is almost invariant with respect to $\f$ and $\mu$.
Moreover, $B$ and $B^c$ are almost invariant with respect to $\f^n$ and $\mu$ for every $n \in N$.
\end {lemma}

\begin{proof}
Since $B$ is almost invariant, there exists a Borel subset $B'$ of $B$
such that $\f^{\dagger}(B') \subseteq B$ and $\mu(B')=\mu(B)$. Consider,
$
C:= \{x \in B^c : \f^{\dagger}(x) \cap B \neq \emptyset\}.
$
Observe that $\f^{\dagger}(B^c\setminus C) \subseteq B^c$. 
Since $C= \f(B) \cap B^c$, $C$ is a Borel set.
Thus if we prove $\mu(C)=0$,
we are done.
If we take $\phi:= \chi_B$, the 
characteristic function of $B$, in \eqref{E:inv_int}, we get
\[
\int_X \chi_B\,d\mu=\int_{B'}  1 \,d\mu
+ \int_{B \setminus B'} \countsum_{y \in \f^{\dagger}(x)} \frac{\chi_B(y)}{\dg}\,d\mu(x)
+ \int_{C} \countsum_{y \in \f^{\dagger}(x)} \frac{\chi_B(y)}{\dg}\,d\mu(x).
\]
By the definition of $C$ and since $\mu(B')=\mu(B)$, we get
\[
\mu(B)\geq \mu(B)
+ \int_{C} \frac{1}{\dg}\,d\mu.
\]
Thus $\mu(C)=0$. This proves that $B^c$ is almost invariant with respect to $\f$ and $\mu$.
\smallskip

We now prove that $B$ is almost invariant with respect to $\f^n$ and $\mu$.
We use induction on $n$ to prove this.
First note that the measure $\mu$ is $(\f^n)^*$-invariant for every $n \in \N$.
Assume that $B$ is almost invariant with respect to $\f^j$ and $\mu$.
We claim that $B$ is almost invariant with respect to $\f^{j+1}$ and $\mu$.
There exists a Borel subset $B_j'$ of $B$
such that ${(\f^j)}^{\dagger}(B_j') \subseteq B$ and $\mu(B_j')=\mu(B)$.
Consider $B_{j+1}':=B_{j}' \setminus \f^j(B \setminus B')$, where $B'$ is a Borel subset of $B$ such that
$\f^{\dagger}(B') \subseteq B$ and $\mu(B')=\mu(B)$.
It follows that ${(\f^{j})}^{\dagger}(B_{j+1}') \subseteq B'$.
Thus
${(\f^{j+1})}^{\dagger}(B_{j+1}') \subseteq B$.
It remains to prove that $\mu(B_{j+1}')=\mu(B)$. 
Since the measure $\mu$ is $(\f^j)^*$-invariant, it follows that, for any Borel subset $A$ of $X$, we have
\[
\frac{1}{\dg^j}\, \mu(\f^j(A)) \leq \mu(A).
\]
See \cite[Lemma~5.6]{Londhe:ridmcahs22} for a detailed proof.
Since $\mu(B \setminus B')=0$, we get
$\mu(\f^j(B \setminus B'))=0$. 
Thus $B$ is almost invariant with respect to $\f^n$ and $\mu$ for every $n \in \N$.
As in the first paragraph of the proof, we show $B^c$ is almost invariant with respect to $\f^n$ and $\mu$
for every $n \in \N$.
\end{proof}

Given a set $A$ such that $\f^{\dagger}(A) \subseteq A$,
we may not have $\f^{\dagger}(A^c) \subseteq A^c$ as in the case of maps.
However, in the presence of an
$\f^*$-invariant measure $\mu$, the above lemma asserts in particular that $A^c$ is
almost invariant with respect to $\f$ and $\mu$.
\smallskip

We first show the existence of an $\f^*$-invariant measure 
when the manifold $X$ is compact.
If $X$ is non-compact, then we cannot guarantee the existence
of an $\f^*$-invariant measure, for example, consider $X=\C$ and the holomorphic correspondence
induced by the graph of the map $f(z)=z+1$.

\begin{proposition}\label{P:invar_exist}
Let $\f$ be a holomorphic correspondence of topological degree $\dg$
on a {\bf{compact}} complex manifold $X$. Then
there exists an $\f^*$-invariant Borel probability measure.
\end{proposition}

\begin{proof}
Let $\{\nu_n\}$ be a sequence of Borel probability measures. Consider
\[
\mu_n:= \frac{1}{n} \sum_{j=0}^{n-1} \frac{(\f^j)^*\nu_n}{\dg^j}.
\]
Since $X$ is compact, the sequence $\{\mu_n\}$ is tight. Thus there exists a convergent subsequence.
Let $\{\mu_{n_j}\}$ be such subsequence with the limit $\mu$. We next
show that $\mu$ is $\f^*$-invariant. For continuous $\psi$, we have
\begin{align}
\Big| \int_X \countsum_{y \in \f^{\dagger} (x)} \frac{\psi(y)}{\dg}&\,d\mu_{n_j}(x) - \int_X \psi(x) \,d\mu_{n_j}(x) \Big| \notag \\
&= \frac{1}{n_j} \Big| \int_X \sum_{i=0}^{n_j-1} 
\Big( \countsum_{y \in ({\f^{i+1}})^{\dagger} (x)} \frac{\psi(y)}{\dg^{i+1}} -\countsum_{y \in {\f^{i}}^{\dagger} (x)} \frac{\psi(y)}{\dg^{i}} \Big) \,d\nu_{n_j}(x)\Big| \notag \\
&= \frac{1}{n_j} \Big| \int_X \Big( \countsum_{y \in ({\f^{n_j}})^{\dagger} (x)} \frac{\psi(y)}{\dg^{n_j}}\Big) - \psi(x) \,d\nu_{n_j}(x)\Big| \notag \\
&\leq \frac{2}{n_j} \|\psi\|_{\infty}. \notag
\end{align}
As $j \to \infty$, the last expression tends to 0. Thus, we get
\[
\int_X \countsum_{y \in \f^{\dagger} (x)} \frac{\psi(y)}{\dg}\,d\mu(x) = \int_X \psi(x) \,d\mu(x).
\]
As $\psi$ is arbitrary, it follows that $\mu$ is $\f^*$-invariant.
\end{proof}

The next proposition
characterises ergodic measures among $\f^*$-invariant Borel probability measures.
As an application of this characterisation, we get the existence of an $\f^*$-invariant 
ergodic measure when $X$ is compact.

\begin{proposition}\label{P:ergodic_exist}
Let $\f$ and $X$ be as in Proposition~\ref{P:invar_exist}.
Then an 
$\f^*$-invariant measure $\mu$ is ergodic if and only if
$\mu$ {\bf{cannot}} be written as a strict convex combination of two distinct $\f^*$-invariant probability measures,
i.e.,
there do not exist $\f^*$-invariant Borel probability measures $\mu_1 \neq \mu_2$
and $0< \lambda <1$ such that 
$\mu= \lambda \mu_1 + (1- \lambda) \mu_2$.
\end{proposition}

\begin{proof}
Assume that $\mu$ is not ergodic. Therefore, there exist a Borel set $B$ such that $B$ is almost invariant
and $0<\mu(B)<1$. By Lemma~\ref{L:alminv_complement}, $B^c$ is almost invariant.
Given a Borel subset $A$ of $X$, let $\mu|_A$
denotes the restriction measure, defined by $\mu|_A(C):=\mu(A \cap C)$ for any Borel
subset $C$ of $X$. It is easy to see that
\[
\frac{1}{\mu(B)} \, \mu|_B \text{ \,\, and \,\, } \frac{1}{\mu(B^c)} \, \mu|_{B^c}
\]
are $\f^*$-invariant Borel probability measures. Moreover,
\[
\mu= \mu(B) \Big(\frac{1}{\mu(B)} \, \mu|_B \Big)+  \mu(B^c) \Big(\frac{1}{\mu(B^c)} \,\mu|_{B^c} \Big).
\]
Thus $\mu$ can be written as a strict convex combination of two distinct $\f^*$-invariant Borel 
probability measures.
\smallskip

Conversely, let $\mu$ be ergodic and assume that $\mu= \lambda \mu_1 + (1- \lambda) \mu_2$
for some $0< \lambda <1$.
Since $\lambda >0$, $\mu_1$ is absolutely continuous with respect to $\mu$.
Thus there is a positive function $\varphi$ such that, for all Borel subsets $A$,
\begin{equation}\label{E:Radon-Nikodym}
\mu_1(A)= \int_A \varphi\,d\mu.
\end{equation}
Let $B:=\{x \in X: \varphi(x)<1\}$. We now prove that $B$ is almost invariant
with respect to $\f$ and $\mu$. Consider the sets
\[
C_1:= \{x \in B: \f^{\dagger}(x) \cap B^c \neq \emptyset\} \text{   and   } 
C_2:=\{x \in B^c: \f^\dagger(x) \cap B \neq \emptyset\}.
\]
Observe that $C_1= \f(B^c) \cap B$ and $C_2= \f(B) \cap B^c$. Thus $C_1$ and $C_2$ are Borel sets.
We claim that $\mu(C_1)=0$.
Since $\mu$ is $\f^*$-invariant, by \eqref{E:inv_int}, we get
\[
\int_X \chi_B\,d\mu
= \int_X \chi_{B\setminus C_1}\,d\mu + \int_{C_1}  \countsum_{y \in \f^{\dagger}(x)} \frac{\chi_B(y)}{\dg}\,d\mu(x)
+ \int_{C_2} \countsum_{y \in \f^{\dagger}(x)} \frac{\chi_B(y)}{\dg}\,d\mu(x)
\]
Therefore,
\begin{equation}\label{E:mu_invariance}
\int_{C_1} 1\,d\mu 
= \int_{C_1}  \countsum_{y \in \f^{\dagger}(x)} \frac{\chi_B(y)}{\dg}\,d\mu(x)
+ \int_{C_2}  \countsum_{y \in \f^{\dagger}(x)} \frac{\chi_B(y)}{\dg}\,d\mu(x).
\end{equation}
Similarly, since $\mu_1$ is $\f^*$-invariant, we also have
\[
\int_{C_1} 1\,d\mu_1
= \int_{C_1} \countsum_{y \in \f^{\dagger}(x)} \frac{\chi_B(y)}{\dg}\,d\mu_1(x)
+ \int_{C_2} \countsum_{y \in \f^{\dagger}(x)} \frac{\chi_B(y)}{\dg}\,d\mu_1(x).
\]
By using \eqref{E:Radon-Nikodym}, we get
\begin{equation}\label{E:mu1_invariance}
\int_{C_1} \varphi\,d\mu 
= \int_{C_1} \Big( \countsum_{y \in \f^{\dagger}(x)} \frac{\chi_B(y)}{\dg}\Big)\varphi(x)\,d\mu(x)
+ \int_{C_2} \Big( \countsum_{y \in \f^{\dagger}(x)} \frac{\chi_B(y)}{\dg}\Big) \varphi(x)\,d\mu(x).
\end{equation}
Subtracting \eqref{E:mu_invariance} from \eqref{E:mu1_invariance} gives
\begin{align}\label{E:equality_ergodic_exist}
\int_{C_1} (\varphi - 1)\,d\mu 
= \int_{C_1} &\Big( \countsum_{y \in \f^{\dagger}(x)} \frac{\chi_B(y)}{\dg}\Big)(\varphi(x)-1)\,d\mu(x) \notag\\
&+ \int_{C_2} \Big( \countsum_{y \in \f^{\dagger}(x)} \frac{\chi_B(y)}{\dg}\Big) (\varphi(x)-1)\,d\mu(x).
\end{align}
Observe that, for every $x \in {C_1}$, $\decosum{y \in \f^{\dagger}(x)} \chi_B(y)/\dg < 1$.
Also,
$\varphi(x)-1 < 0$ for every $x \in {C_1}$, and $\varphi(x)-1 \geq 0$ for every $x \in C_2$.
Thus,
if $\mu(C_1)>0$, then \eqref{E:equality_ergodic_exist} does not hold.
%holds only if $\mu(C_1)=0$ and $\mu(C_2)=0$. 
This proves that $B$ is almost invariant
with respect to $\f$ and $\mu$.
Since the measure $\mu$ is ergodic, either $\mu(B)=0$ or $\mu(B)=1$. If $\mu(B)=1$ then,
by \eqref{E:Radon-Nikodym}, we get
\[
\mu_1(X)= \int_X \varphi\,d\mu < \mu(B)=1,
\]
which gives a contradiction to the fact that $\mu_1$ is a probability measure. Thus $\mu(B)=0$.
Similarly, it can be shown that $\mu(\{x \in X: \varphi(x)>1\})=0$. Therefore, $\varphi$ is $\mu$-almost everywhere
equal to 1. By \eqref{E:Radon-Nikodym}, $\mu = \mu_1$, and consequently, $\mu = \mu_1=\mu_2$.
Therefore, $\mu$ cannot be written as a strict convex combination of two 
distinct $\f^*$-invariant Borel probability measures.
\end{proof}

Let $\f$ be a holomorphic correspondence on a compact complex manifold $X$.
We now use Proposition~\ref{P:ergodic_exist} to show the existence of an $\f^*$-invariant ergodic measure.
Observe that the set $\mathcal{M}_\f$ of $\f^*$-invariant Borel probability measures is a non-empty compact convex set. 
Thus, by
the Krein--Milman theorem, it follows that the set of extreme points of $\mathcal{M}_\f$ is nonempty.
By Proposition~\ref{P:ergodic_exist},
it follows that an $F^*$-invariant ergodic measure always exists when $X$ is compact.
\medskip

\section{Examples}\label{S:examples}

Let $\f$ be a holomorphic correspondence on a complex manifold $X$.
In Section~\ref{S:existence}, we showed that an $\f^*$-invariant and an $\f^*$-invariant ergodic measure 
always exist
when the manifold $X$ is compact.
This section is devoted to explicit examples of $\f^*$-invariant ergodic measures.
In particular, we prove Theorem~\ref{T:ergodicity_DS}, which gives examples
of holomorphic correspondences with dynamically interesting
ergodic measures.
We begin by giving an example alluded to in Section~\ref{S:intro}.
If $\f$ is a holomorphic map, then the definition of the almost invariance of a Borel set $B$ with
respect to $\f$ and $\mu$ is equivalent to
$\mu(B \triangle \f^{\dagger}(B))=0$. In general, for holomorphic correspondences, this is not the case. We now
give an example of such holomorphic correspondence.
 
 \begin{example}\label{Ex:sym_diff}
An example of a holomorphic correspondence demonstrating that an almost invariant Borel set $B$
need not satisfy $\mu(B \triangle \f^{\dagger}(B))=0$.
\end{example}

\noindent{Consider a finitely generated rational semigroup $S$ generated by
$\mathcal G=\{z^2, z^2 / 2\}$. Let $\J(S)$ denote the Julia set of $S$.} Then (see \cite[Example~1]{boyd:imfgrs99}),
\[
\J(S)=\{z \in \C: 1 \leq |z| \leq 2\}.
\]
Let $\mu_{\mathcal G}$ be the measure constructed in \cite{boyd:imfgrs99}.
See \cite[Example~1]{boyd:imfgrs99}) for the explicit formula of the measure $\mu_{\mathcal G}$.
It turns out that $\mu_{\mathcal G}$ is the Dinh--Sibony measure associated with the holomorphic 
correspondence $\f_{\mathcal G}$ in \eqref{E:semi_corres} induced by $\mathcal G$. 
Since $\supp(\mu_{\mathcal G})=\J(S)$,
we have $\supp(\mu_{\mathcal G})=\{z \in \C: 1 \leq |z| \leq 2\}$.
Now consider a Borel set $B:=\{z \in \C: |z| > 2\}$. Since $\mu_{\mathcal G}(B)=0$, it follows that
$B$ is almost invariant with respect to $\f_{\mathcal G}$ and $\mu_{\mathcal G}$.
Observe that $\f_{\mathcal G}^{\dagger} (B)=\{z \in \C: |z| > \sqrt2\}$.
Since $\supp(\mu_{\mathcal G})=\{z \in \C: 1 \leq |z| \leq 2\}$, it follows that
$\mu_{\mathcal G}(\f_{\mathcal G}^{\dagger} (B) \triangle B) \neq 0$.
\hfill $\blacktriangleleft$
\medskip

We now prove that the Dinh--Sibony measure 
given by Result~\ref{R:DS} is ergodic as in Definition~\ref{D:ergodic}. This gives explicit examples 
of $\f^*$-invariant ergodic measures.
In particular, the measure $\mu_{\mathcal G}$ in Example~\ref{Ex:sym_diff} is ergodic.
Observe that, the measure $\mu_{\mathcal G}$ in Example~\ref{Ex:sym_diff} is not invariant
in the classical sense of ergodic theory.
\begin{proof}[The proof of Theorem~\ref{T:ergodicity_DS}]
As noted in Section~\ref{S:def},
there exists a pluripolar subset $E$ of $X$ such that for every $a \in X \setminus E$, we have
\[
  \dg^{-n} {(\f^n)}^* {\delta}_a\to \mu_{\f}
\] 
as $n \to \infty$. Equivalently, if $\phi$ is a continuous function on $X$, then
\begin{equation}\label{E:dual}
\frac{1}{\dg^{n}} \countsum_{y \in {\f^{n}}^{\dagger} (a)} \phi(y)\to \int_X \phi \,d\mu_{\f}
\end{equation}
as $n \to \infty$ for every $a \in X \setminus E$.
Let $\psi$ be $\mu_\f$-integrable.
Since $\mu_\f (E)=0$, by \eqref{E:dual} and by the dominated convergence theorem, we get
\begin{equation}\label{E:conver}
\int_X \bigg(\frac{1}{\dg^{n}} \countsum_{y \in {\f^{n}}^{\dagger} (x)} \phi(y) \bigg) \psi(x) \,d \mu_{\f}(x)
\to \int_X \phi \,d\mu_{\f}  \int_X \psi \,d\mu_{\f}
\end{equation}
as $n \to \infty$. If $\psi$ is bounded, then by the density of the continuous functions, \eqref{E:conver} holds 
if $\phi$ is $\mu_\f$-integrable.
\smallskip

Consider a Borel subset $B$ of $X$ that is
almost invariant with respect to $\f$ and $\mu$. Let $\phi= \chi_B$. 
By Lemma~\ref{L:alminv_complement}, $B$ and $B^c$ are almost invariant with respect to $\f^n$
and $\mu$ for every $n \in \N$. Thus it follows that, for all $n \in \N$,
\[
\frac{1}{\dg^{n}} \countsum_{y \in {\f^{n}}^{\dagger} (x)} \phi(y)= \phi(x)
\]
for $\mu_\f$-almost every $x\in X$. Let $\psi=1-\chi_B$. By \eqref{E:conver}, we get
$\mu_\f(B) \cdot \mu_\f(B^c)=0$. Thus we have either $\mu_\f(B)=0$ or $\mu_\f(B^c)=0$.
Therefore, the Dinh--Sibony measure $\mu_{\f}$ is ergodic as in Definition~\ref{D:ergodic}.
\end{proof}

\begin{remark}\label{Re:mero_corres}
Given a meromorphic correspondences $\f$ on a complex manifold $X$ and a Borel subset $B$ of $X$,
$\f^{\dagger}(B)$ need not be a Borel subset of $X$
(see \cite{Londhe:ridmcahs22} for examples).
This difficulty can be handled using tools from descriptive set theory.
Using these tools further as in \cite{Londhe:ridmcahs22},
Theorem~\ref{T:ergodic} can be proved when $\f$ is defined on a compact manifold
and for an $\f^*$ invariant measure that
puts zero mass on pluripolar sets. The above proof of Theorem~\ref{T:ergodicity_DS} 
also holds with appropriate changes for the meromorphic case.
\end{remark}

\begin{remark}
Using Theorem~2.10 in \cite{Dinh:ddppdcp05} and the technique as in the last paragraph of the proof of Theorem~\ref{T:ergodicity_DS}, 
it follows that the measures constructed in \cite{Dinh:ddppdcp05} are also ergodic, as in Definition~\ref{D:ergodic}.
\end{remark}

\section{Preliminary results}\label{S:prelem}

This section is devoted to proving certain lemmas that are essential for the proof of Theorem~\ref{T:ergodic}. 
We begin by proving a lemma that is fundamental
and will be used multiple times in this paper.
\begin{lemma}\label{L:imp}
Let $\f$ be a holomorphic correspondence of topological degree $\dg$ on a complex manifold
$X$ and $\mu$ an $\f^*$-invariant Borel probability measure on $X$. If $\phi: X \to \C$ in $L^1(\mu)$ is
{\bf{real valued}} and satisfies
\[
\countsum_{y \in \f^{\dagger}(x)} \frac{\phi(y)}{\dg}=\phi(x)
\]
for $\mu$-almost every $x \in X$, then the sets $\{x \in X : \phi(x) > t\}$ and $\{x \in X : \phi(x) < t\}$ are almost
invariant with respect to $\f$ and $\mu$
for any $t \in \R$.
\end{lemma}

\begin{proof}
We first prove that $B:=\{x \in X : \phi(x) > 0\}$ is almost invariant with respect to $\f$ and $\mu$.
Let
\[
C_1:=\{x \in B : \f^{\dagger}(x) \cap B^c \neq \emptyset\} \text{   and   }
C_2:=\{x \in B^c : \f^{\dagger}(x) \cap B \neq \emptyset\}.
\]
Take $\phi:= \phi \cdot \chi_{B}$ in \eqref{E:inv_int} to get
\begin{align}\label{E:equality}
\int_{X} \phi\cdot \chi_{B}\,d\mu= \int_{B \setminus C_1} \countsum_{y \in \f^{\dagger}(x)} \frac{\phi(y)}{\dg}\,d\mu&(x)
+\int_{C_1} \countsum_{y \in \f^{\dagger}(x)\,\cap B} \frac{\phi(y)}{\dg}\,d\mu(x) \notag \\
&+\int_{C_2} \countsum_{y \in \f^{\dagger}(x)\,\cap B} \frac{\phi(y)}{\dg}\,d\mu(x). \notag
\end{align}
Since $\phi$ satisfies $\decosum{y \in \f^{\dagger}(x)} \phi(y)/\dg=\phi(x)$ for $\mu$-almost every $x \in X$, we have
\begin{align}
\int_{B} \phi\,d\mu = \int_{B \setminus C_1} \phi\,d\mu %+ \int_{C_2} \phi\,d\mu
+ \int_{C_1}& \Big(\phi(x) - \countsum_{y \in \f^{\dagger}(x)\,\cap B^c} \frac{\phi(y)}{\dg}\Big)\,d\mu(x) \notag\\
&+\int_{C_2} \countsum_{y \in \f^{\dagger}(x)\,\cap B} \frac{\phi(y)}{\dg}\,d\mu(x). \notag
\end{align}
Therefore, we get
\[
\int_{C_1} \countsum_{y \in \f^{\dagger}(x)\,\cap B^c} \frac{\phi(y)}{\dg}\,d\mu(x)
=\int_{C_2} \countsum_{y \in \f^{\dagger}(x)\,\cap B} \frac{\phi(y)}{\dg}\,d\mu(x).
\]
If $y \in \f^{\dagger}(x) \cap B^c$, then $\phi(y)\leq0$, and if $y \in \f^{\dagger}(x) \cap B$, then $\phi(y)>0$. 
Thus the above equality holds only if $\mu(C_2)=0$.
This proves that $B^c$ is almost invariant with respect to $\f$ and $\mu$.
By invoking Lemma~\ref{L:alminv_complement}, $B$ is almost invariant with respect to $\f$ and $\mu$.
\smallskip

Fix $t \in \R$. Observe that the function $\phi-t$ satisfies
\[
\countsum_{y \in \f^{\dagger}(x)} \frac{(\phi-t)(y)}{\dg}=\Big(\countsum_{y \in \f^{\dagger}(x)} \frac{\phi(y)}{\dg}\Big)-t=(\phi-t)(x)
\]
for $\mu$-almost every $x\in X$. Thus $\{x \in X : \phi(x)-t > 0\}= \{x \in X : \phi(x) > t\}$ is almost invariant.
Since $t$ is arbitrary, it follows that
$\{x \in X : \phi(x) > t\}$ is almost invariant for every $t \in \R$.
It only remains to show that $\{x \in X : \phi(x) < t\}$ is almost invariant. Note that
$\{x \in X : \phi(x) < t\}= \{x \in X : -\phi(x) > -t\}$ and $-\phi$ satisfies
\[
\countsum_{y \in \f^{\dagger}(x)} \frac{-\phi(y)}{\dg}=-\phi(x)
\]
for $\mu$-almost every $x\in X$. Thus $\{x \in X : \phi(x) < t\}$ is also almost invariant 
with respect to $\f$ and $\mu$ for every $t \in \R$.
\end{proof}

We now use the above lemma to characterise ergodic measures in terms of 
functions satisfying certain invariance property.

\begin{lemma}\label{L:ergodic_equi}
Let $\f$ be a holomorphic correspondence of topological degree $\dg$ on a complex manifold
$X$ and $\mu$ an $\f^*$-invariant Borel probability measure on $X$. The measure
$\mu$ is ergodic if and only if for $\phi: X \to \C$ in $L^1(\mu)$, 
\[
\countsum_{y \in \f^{\dagger}(x)} \frac{\phi(y)}{\dg}=\phi(x)
\]
for $\mu$-almost every $x\in X$ implies that $\phi$ is equal to a constant $\mu$-almost everywhere.
\end {lemma}

\begin{proof}
Consider a Borel set $B$ that is almost invariant with respect to $\f$ and $\mu$. 
Let $\phi:=\chi_{B}$, the 
characteristic function of $B$. Since $B$ is almost invariant, we have
\[
\countsum_{y \in \f^{\dagger}(x)} \frac{\chi_{B}(y)}{\dg}=1=\chi_{B}(x)
\]
for $\mu$-almost every $x \in B$.
By Lemma~\ref{L:alminv_complement}, $B^c$ is also almost
invariant with respect to $\f$ and $\mu$. Thus
\[
\countsum_{y \in \f^{\dagger}(x)} \frac{\chi_{B}(y)}{\dg}=0=\chi_{B}(x)
\]
for $\mu$-almost every $x \in B^c$.
Owing to the hypothesis, $\chi_{B}$ is equal to a constant $\mu$-almost everywhere.
Therefore, either $\mu(B)=0$ or $\mu(B)=1$.
\smallskip

Now, we prove the sufficiency part. 
We consider real and imaginary parts separately.
Observe that the real part of $\phi$, $\Re(\phi)$, satisfies
\[
\countsum_{y \in \f^{\dagger}(x)} \frac{\Re(\phi)(y)}{\dg}=\Re(\phi)(x)
\]
for $\mu$-almost every $x\in X$.
By Lemma~\ref{L:imp}, for any $t \in \R$, $A_t:=\{x \in X: \Re(\phi)(x) >t\}$ is almost invariant
with respect to $\f$ and $\mu$.
Since $\mu$ is ergodic, for any $t \in \R$, either $\mu(A_t)=0$ or $\mu(A_t)=1$. Let $t'$
be the minimum value of $t$ such that $\mu(A_{t'})=0$. It easily follows that $\Re(\phi)$ is equal to the constant $t'$
$\mu$-almost everywhere. Similarly, it can be shown that the imaginary part of $\phi$ is equal to 
a constant $\mu$-almost everywhere.
This proves that the function $\phi$ is equal to a constant $\mu$-almost everywhere.
\end{proof}

We next state a result that will be useful in proving a version of the maximal inequality in our setting.
A proof can be found in \cite[Theorem~1.16]{Walters:aiet82}.

\begin{result}\label{R:max_ineq}
Let $U: L^1(\nu) \to L^1(\nu)$ be a positive linear operator with $\|U\| \leq 1$. For $\psi \in L^1(\nu)$ 
a real valued function, define $\psi_0=0$, $\psi_n=\psi+U\psi+U^2\psi+\dots+U^{n-1}\psi$ for $n \geq 1$.
Let $N>0$ be an integer and $\Psi_N=\max_{0\leq n \leq N} \psi_n$. Then
\[
\int_{\{x:\,\Psi_N(x)>0\}} \psi\,d\nu \geq 0.
\]
\end{result}

We use the above result to prove a version of the maximal inequality in our setting.
Let $\f$ be a holomorphic correspondence and $\mu$ an $\f^*$-invariant measure.
We apply Result~\ref{R:max_ineq} to the operator $U_\f:L^1(\mu)\to L^1(\mu)$
defined by
\[
U_\f(\varphi)(x):=\countsum_{y \in \f^{\dagger}(x)} \frac{\varphi(y)}{\dg}.
\]
Observe that $U_\f$ is a positive linear operator with $\|U_\f\|=1$.

\begin{corollary}\label{C:max_ergodic}
Let $\f$ be a holomorphic correspondence of topological degree $\dg$ on a complex manifold
$X$ and $\mu$ an $\f^*$-invariant Borel probability measure on $X$.
If $\phi \in L^1({\mu})$ is real valued and
\[
E_{\alpha}:=\Big\{x \in X: \sup_{n \geq 1} \frac{1}{n} \sum_{j=0}^{n-1} \countsum_{y \in {\f^j}^{\dagger}(x)} \frac{\phi(y)}{\dg^j} > \alpha\Big\},
\]
then
\[
\int_{E_{\alpha} \cap A} \phi\,d\mu \geq \alpha\,\mu(E_{\alpha} \cap A),
\]
if $A$ is an almost invariant set with respect to $\f$ and $\mu$.
\end{corollary}

\begin{proof}
Let $\psi=\phi-\alpha$.
In the notations of Result~\ref{R:max_ineq},
\[
E_{\alpha}= \bigcup_{N=0}^{\infty} \{x: \Psi_N(x) >0\}.
\]
By Result~\ref{R:max_ineq}, it follows that $\int_{E_{\alpha}} \psi  \,d\mu\geq 0$ and therefore
$\int_{E_{\alpha}} \phi \,d\mu \geq \alpha \mu(E_{\alpha})$. Now, if $A$ is an almost invariant set
with respect to $\f$ and $\mu$,
then we apply the same argument to $\psi=\phi-\alpha$ and the $\f^*$-invariant probability measure
$\frac{1}{\mu(A)}\, \mu|_A$, to get
\[
\int_{E_{\alpha} \cap A} \phi\,d\mu \geq \alpha\, \mu(E_{\alpha} \cap A).
\]
This finishes the proof.
\end{proof}

\section{The proof of Theorem~\ref{T:ergodic}}\label{S:ergodic}

Before starting the proof of Theorem~\ref{T:ergodic}, we recall a notation. Recall that
${\sum\nolimits}'$ denotes the sum with $y$'s, repeated with multiplicity.
A careful reader will observe that the proofs in Section~\ref{S:prelem} (and the proof of Theorem~\ref{T:ergodic})
are purely measure theoretic and the complex structure does not play any role in the proofs. 
Thus these results can be extended to more general multi-valued maps.
We are now ready for

\begin{proof}[The proof of Theorem~\ref{T:ergodic}]
Note that, by considering real and imaginary parts separately, it is enough to consider only real valued $\phi$.
Define, for $x \in X$,
\[
\phi'(x):=\liminf_{n\to \infty}\frac{1}{n}\sum_{j=0}^{n-1}\countsum_{y \in {{\f}^j}^{\dagger}(x)} \frac{\phi(y)}{\dg^{j}},
 \]
\[
\phi''(x):=\limsup_{n\to \infty}\frac{1}{n}\sum_{j=0}^{n-1}\countsum_{y \in {{\f}^j}^{\dagger}(x)} \frac{\phi(y)}{\dg^{j}}.
\]
We now prove that $\decosum{y \in \f^{\dagger}(x)}{\phi'(y)/\dg}=\phi'(x)$ and 
$\decosum{y \in \f^{\dagger}(x)} {\phi''(y)/\dg}=\phi''(x)$ hold for $\mu$-almost every $x\in X$. Observe that
\begin{equation}\label{E:liminf_limsup}
\frac{n+1}{n} \bigg(\frac{1}{n+1}\sum_{j=0}^{n}\countsum_{y \in {{\f}^j}^{\dagger}(x)} \frac{\phi(y)}{\dg^j}\bigg)
=\countsum_{y \in \f^{\dagger}(x)} \frac{1}{\dg}\bigg(\frac{1}{n}\sum_{j=0}^{n-1}\countsum_{z \in {F^j}^{\dagger}(y)} \frac{\phi(z)}{\dg^j} \bigg) + \frac{1}{n} \phi(x).
\end{equation}
By taking the limit along a subsequence for which the left-hand side of \eqref{E:liminf_limsup} 
converges to the liminf, gives us $\phi'(x) \geq \decosum{y \in \f^{\dagger}(x)} {\phi'(y)/\dg}$
holds for $\mu$-almost every $x\in X$. 
The limit along a subsequence for which the right-hand side of \eqref{E:liminf_limsup} 
converges to the liminf, gives us $\phi'(x) \leq \decosum{y \in \f^{\dagger}(x)} {\phi'(y)/\dg}$
holds for $\mu$-almost every $x\in X$.
A similar argument for $\phi''$, gives us the desired equalities.
\smallskip

We next prove that $\phi' = \phi''$ $\mu$-almost everywhere.
For rationals $\alpha > \beta$, define
\[
E_{\alpha, \beta}:=\{x \in X: \phi' (x) < \beta \text{ and } \phi''(x) >\alpha\}.
\]
Since $\decosum{y \in \f^{\dagger}(x)} \phi'(y)/\dg=\phi'(x)$ and 
$\decosum{y \in \f^{\dagger}(x)} {\phi''(y)/\dg}=\phi''(x)$ hold for $\mu$-almost every $x\in X$,
by Lemma~\ref{L:imp},
it follows that $\{x \in X: \phi'(x) < \beta\}$ and $\{x \in X: \phi''(x) >\alpha\}$ are almost invariant.
It is now easy to see that $E_{\alpha, \beta}$ is almost invariant. 
Now, we apply Corollary~\ref{C:max_ergodic} to the setting here. Observe that 
$E_{\alpha, \beta} \subseteq E_{\alpha}$,
where $E_{\alpha}$ is the set defined in Corollary~\ref{C:max_ergodic}.
Since $E_{\alpha, \beta}$ is almost invariant, we have
\[
\int_{E_{\alpha, \beta}} \phi\,d\mu \geq \alpha \mu(E_{\alpha, \beta}).
\]
Similar argument, by replacing $\phi$ by $-\phi$, shows that
\[
\int_{E_{\alpha, \beta}} \phi\,d\mu \leq \beta \mu(E_{\alpha, \beta}).
\]
The above two inequalities show that $\mu(E_{\alpha, \beta})=0$ for $\alpha > \beta$. Since 
$\{x \in X: \phi' (x) < \phi''(x)\}= \cup\,\{E_{\alpha, \beta}: \alpha > \beta \text{ with } \alpha, \beta \in \Q\}$,
it follows that $\phi' = \phi''$ $\mu$-almost everywhere. Set $\Phi:=\phi'$. Therefore,
\[
\lim_{n\to \infty}\frac{1}{n}\sum_{j=0}^{n-1}\countsum_{y \in {{\f}^j}^{\dagger}(x)} \frac{\phi(y)}{\dg^j}
= \Phi(x) \quad \text{for }\mu\text{-almost every }x\in X.
 \]
 \smallskip
 
Next, we show that $\Phi \in L^1({\mu})$, as a simple application of Fatou's Lemma. It is easy see that
\[
|\Phi(x)| \leq \liminf_{n \to \infty}\frac{1}{n}\sum_{j=0}^{n-1}\countsum_{y \in {{\f}^j}^{\dagger}(x)} \frac{|\phi(y)|}{\dg^j}.
\]
By Fatou's Lemma and, since $\mu$ is $F^*$-invariant, we get
\[
\int_X |\Phi|\,d\mu \leq 
\liminf_{n \to \infty} \int_X \frac{1}{n}\sum_{j=0}^{n-1}\countsum_{y \in {{\f}^j}^{\dagger}(x)} \frac{|\phi(y)|}{\dg^j}\,d\mu
= \int_X |\phi| \,d\mu.
\]
Since $\phi \in L^1(\mu)$, it follows that $\Phi \in L^1({\mu})$.
\smallskip

It only remains to show that $\int_X \Phi\,d\mu=\int_X \phi\,d\mu$. Define, for $k \in \Z$ and $n \geq 1$,
\[
D_k^n:=\{x \in X : \frac{k}{n} \leq \phi''(x) < \frac{k+1}{n}\}.
\]
Observe that, for $\epsilon >0$ small, we have $D_k^n \subseteq E_{(k/n)-\epsilon}$, where $E_{(k/n)-\epsilon}$
is the set defined in Corollary~\ref{C:max_ergodic}. Also, note that 
$D_k^n=\{x \in X :  \phi''(x) \geq k/n\} \cap \{x \in X :\phi''(x) < (k+1)/{n}\}$.
By Lemma~\ref{L:alminv_complement} and Lemma~\ref{L:imp}, it follows that $D_k^n$ is almost invariant.
By invoking Corollary~\ref{C:max_ergodic}, we see that
\[
\int_{D_k^n} \phi \,d\mu \geq \Big(\frac{k}{n} - \epsilon \Big) \mu (D_k^n).
\]
Since $\epsilon >0$ is arbitrary, we have
\[
\int_{D_k^n} \phi \,d\mu \geq \frac{k}{n} \, \mu (D_k^n).
\]
Now, by the definition of $D_k^n$ and the last inequality,
\[
\int_{D_k^n} \phi'' \,d\mu \leq \frac{k+1}{n} \, \mu (D_k^n) \leq \frac{1}{n} \, \mu (D_k^n) + \int_{D_k^n} \phi \,d\mu.
\]
Summing over $k$ gives us
\[
\int_{X} \phi'' \,d\mu \leq \frac{\mu (D_k^n)}{n}+ \int_{X} \phi \,d\mu.
\]
Since this holds for all $n \geq 1$, we have 
$
\int_{X} \phi'' \,d\mu \leq \int_{X} \phi \,d\mu.
$
Applying this to $-\phi$ instead of $\phi$, we get
$
\int_{X} \phi' \,d\mu \geq \int_{X} \phi \,d\mu.
$
Since $\phi'=\phi''$ $\mu$-almost everywhere, we have
\[
\int_{X} \Phi\,d\mu = \int_{X} \phi' \,d\mu 
=\int_{X} \phi \,d\mu.
\]
\smallskip

Lastly, if $\mu$ is ergodic as in Definition~\ref{D:ergodic}, then, by Lemma~\ref{L:ergodic_equi},
$\Phi$ is a constant $\mu$-almost everywhere
and the constant is precisely $\int_X \phi\,d\mu$.
\end{proof}

\section*{Acknowledgments}
\noindent{The author would like to thank Norm Levenberg and Gautam Bharali for their helpful comments 
on an earlier draft of this paper. The author also thanks the referee for their helpful
suggestions.}

\end{document}